\documentclass[11pt]{article}
\usepackage{graphicx} 
\usepackage{a4wide}
\usepackage{amsfonts}
\usepackage{amsmath}
\usepackage[unicode=true,psdextra]{hyperref}
\usepackage{amsthm}
\usepackage{graphicx}
\usepackage{comment}
\title{The generalised energy identity and length of necks for $\varepsilon$-harmonic maps}
\author{Andrew M. Roberts\footnote{School of Mathematics, University of Leeds, Leeds, LS2 9JT, United Kingdom\\ A.M.Roberts@leeds.ac.uk}}

\newtheorem{theorem}{Theorem}[section]

\newtheorem{lemma}[theorem]{Lemma}

\newtheorem{corollary}[theorem]{Corollary}

\usepackage{xcolor}
\usepackage{hyperref}
\hypersetup{
    colorlinks=true,
    linkcolor=olive,
    filecolor=magenta,
    urlcolor=orange,
    citecolor=purple
    }

\begin{document}

\maketitle
\begin{abstract}
    In this paper we find analogues for $\varepsilon$-harmonic maps to the generalised energy identity and the existence of geodesic necks result discovered by Yuxiang Li and Youde Wang for $\alpha$-harmonic maps. In particular there exist specific quantities depending only on $\varepsilon$ and the bubbling radius which entirely determine if the full energy identity holds and if a neck forms. In the case these fail we can calculate the energy lost and the length of the geodesic neck based on only these quantities and the biharmonic energy of the bubble.
\end{abstract}

\section{Introduction}

The Dirichlet energy is one of the most widely studied objects in geometric analysis. Let $(M^2,g)$, $(N^n,h)$ be Riemannian manifolds and $u\in W^{1,2}(M,N)$. The Dirichlet energy is defined as follows
\begin{equation*}
    E[u]=\int_M|\nabla u|^2.
\end{equation*}
Critical points of this energy are known as harmonic maps.

Both regularity theory and existence are difficult for this functional.
In order to deal with the existence problem Sacks and Uhlenbeck introduced in \cite{Sacks_Uhlenbeck} the notion of $\alpha$-harmonic maps which arise as critical points of a perturbed energy functional. Regularity for these maps is much easier and the functional does indeed satisfy the Palais-Smale conditions, so in particular energy minimisers within homotopy classes exist. As $\alpha$ tends to 1 their perturbed functional returns to the Dirichlet energy so if we take a sequence of $\alpha_k$-harmonic maps $u_k$ with $k\rightarrow0$ we hope that this will converge to a fully harmonic map. This does indeed happen weakly in $W^{1,2}$ and in $C^\infty_{\text{loc}}$ away from a finite number of points. At these points we observe the standard bubbling phenomenon, when rescaled we obtain harmonic maps from the sphere.

In \cite{Lamm_epsilon} Lamm introduced the notion of $\varepsilon$-energy for $u\in W^{2,2}(M,N)$ defined as
\begin{equation}
    E_{\varepsilon}[u]=\int_M|\nabla u|^2+\varepsilon|\Delta u|^2.
\end{equation}
We call $\varepsilon$-harmonic maps critical points of this energy. This again satisfies Palais-Smale and has nice regularity theory. Note that the Laplacian here is the extrinsic Laplacian obtained by viewing $u$ as a map $M\rightarrow N\hookrightarrow\mathbb{R}^l$ and so the energy depends on the embedding into Euclidean space chosen for $N$. Using the tension field instead would not give us any benefits as harmonic maps would still be critical points of our `intrinsic-$\varepsilon$-energy'. This offers the added complications from being a higher order term and the Euler-Lagrange equation becomes a 4th order equation, it however has the added benefit of linearity which considerably eases some of the computations.

One of the questions we can ask is what happens to the energy in the limit. The goal is to show that no energy is lost, in other words the energy in the limit is totally captured by the energy of the limit map plus the energy of the "bubble maps".
This is known to be true in some cases and false in other cases. In the case of harmonic maps this is true, as shown by Parker \cite{Parker_bubble}. In both the $\alpha$ and $\varepsilon$ cases it is true when the target is a homogeneous manifold, see \cite{Lamm_epsilon}, \cite{li_wang_general_EI}, \cite{li_zhu} and \cite{bayer_roberts}. The energy identity has also been shown to be true for min-max sequences of both $\varepsilon$- and $\alpha$-harmonic maps by Lamm in \cite{Lamm_minmax}. However in general, the energy identity can fail, see \cite{Li_Wang_energy_counterexample} and \cite{bayer_roberts}, these constructions work by varying homotopy class.

For $\alpha$-harmonic maps Li and Wang, \cite{li_wang_general_EI}, discovered a generalised energy identity. This in particular gives them a critical criterion for the energy identity to hold.
\begin{theorem}\label{theorem: generalised EI alpha alpha}
    Let $u_k\in C^\infty(M,N)$ be a sequence of $\alpha_k$-harmonic maps with $0<\alpha_k\rightarrow1$ and $E_{\alpha_k}[u_k]$ uniformly bounded. There exists a finite set of points of energy concentration, $\{x_i\}_{i=1}^I$ and $u_\infty\in C^\infty(M,N)$ a harmonic map such that
    \begin{equation*}
        u_k\rightarrow u_\infty \text{  in }C^l_{\text{loc}}(M\setminus\{x_i\}_{i=1}^I)\text{ for all }l\geq0.
    \end{equation*}
    Now, for each $i$, fix some local isothermal coordinates in some ball $B_{\lambda_i}(x_i)\subset M$,
    \begin{equation*}
        F_i:\big(B_{\lambda_i}(x_i),g\big)\rightarrow \big(\mathbb{D}_{\gamma_i}(0),e^{2\rho_i(x)}(dx_1^2+dx_2^2)\big),
    \end{equation*}
    where each $F_i$ is a smooth isomorphism with $F_i(x_i)=0$ and $\rho_i\in C^\infty(\mathbb{D}_{\gamma_i}(0))$. 
    Then there exists a finite number of bubbles $\{\omega_{i,j}\}_{j=1}^{J(i)}\in C^\infty(\mathbb{S}^2,N)$ non trivial harmonic maps, $x_{i,j,k}\in \mathbb{D}_{\gamma_i}(0)$ sequences of points and $r_{i,j,k}\in\mathbb{R}^{>0}$ such that
    \begin{align*}
        u_k(F_i^{-1}(x_{i,j,k}+r_{i,j,k}\cdot))\rightarrow\omega_{i,j}(\pi^{-1}(\cdot))&\text{  in }C^l_{\text{loc}}(B_R(0)\setminus S_{i,j})\text{ for all }l\geq0,R\geq0,\\
        x_{i,j,k}\rightarrow 0,&\\
        r_{i,j,k}\rightarrow0,&\\
        \textnormal{max}\{\frac{r_{i,j_1,k}}{r_{i,j_2,k}},\frac{r_{i,j_2,k}}{r_{i,j_1,k}},\frac{|x_{i,j_1,k}-x_{i,j_2,k}|}{r_{i,j_1,k}+r_{i,j_2,k}}\}\rightarrow\infty&\text{ whenever }j_1\neq j_2,
    \end{align*}
    where $\pi:\mathbb{S}^2\rightarrow \mathbb{R}^2$ is the stereographic projection and $S_{i,j}$ is the finite, possibly empty, set of points where further bubbles form on $\omega_{i,j}$.
    We then have the energy identity
    \begin{equation*}
        \lim_kE_{\alpha_k}[u_k]=E[u_\infty]+\sum_{i}^I\sum_j^{J(I)}\mu_{i,j}^2 E[\omega_{i,j}],
    \end{equation*}
    where
    \begin{equation*}
        \mu_{i,j}=\lim_kr_{i,j,k}^{1-\alpha_k}.
    \end{equation*}
\end{theorem}

So the energy in the limit only depends on the limiting map, the bubble and our special ratios $\mu_{i,j}$.
This also says that the full energy identity holds if and only if all of our $\mu_{i,j}\rightarrow1$. We note here that $r_{i,j,k}$ and $\omega_{i,j}$ are only unique up to rescaling but a different choice will not affect $E[\omega_{i,j}]$ or $\mu_{i,j}$.

We would now like to try and find an equivalent formulation for $\varepsilon$-harmonic maps. In \cite{Lamm_minmax} Lamm shows that the energy identity holds along min-max sequences of $\varepsilon$-harmonic maps. In particular, they show that the energy identity holds if
\begin{equation*}
    \lim_{k\rightarrow\infty}\varepsilon_k \log(1/\varepsilon_k)\int_M|\Delta u_k|^2\rightarrow0.
\end{equation*}
In the proof of this they in fact showed a stronger formulation of the result with $\log(1/\varepsilon_k)$ replaced by $\log(1/r_k)$, for some choice of the bubbling radius. By examining the blow up we know that $\int_M|\Delta u_k|^2\sim \frac{1}{r_k^2}$ so one now has that the energy identity holds if
\begin{equation*}
    \lim_{k\rightarrow\infty}\frac{\varepsilon_k}{r_k^2}\log(1/ r_k)\rightarrow0.
\end{equation*}
We find that this stronger statement is actually an if and only if.
\begin{theorem}\label{theorem: generalised EI epsilon}
    Let $u_k\in C^\infty(M,N)$ be a sequence of $\varepsilon_k$-harmonic maps with $0<\varepsilon_k\rightarrow0$ and $E_{\varepsilon_k}[u_k]$ uniformly bounded. There exists a finite set of points of energy concentration, $\{x_i\}_{i=1}^I$ and $u_\infty\in C^\infty(M,N)$ a harmonic map such that
    \begin{equation*}
        u_k\rightarrow u_\infty \text{  in }C^l_{\text{loc}}(M\setminus\{x_i\}_{i=1}^I)\text{ for all }l\geq0.
    \end{equation*}
    Now, for each $i$, fix some local isothermal coordinates in some ball $B_{\lambda_i}(x_i)\subset M$,
    \begin{equation*}
        F_i:\big(B_{\lambda_i}(x_i),g\big)\rightarrow \big(\mathbb{D}_{\gamma_i}(0),e^{2\rho_i(x)}(dx_1^2+dx_2^2)\big),
    \end{equation*}
    where each $F_i$ is a smooth isomorphism with $F_i(x_i)=0$ and $\rho_i\in C^\infty(\mathbb{D}_{\gamma_i}(0))$. 
    Then there exists a finite number of bubbles $\{\omega_{i,j}\}_{j=1}^{J(i)}\in C^\infty(\mathbb{S}^2,N)$ non trivial harmonic maps, $x_{i,j,k}\in\mathbb{D}_{\gamma_i}$ sequences of points and $r_{i,j,k}\in\mathbb{R}^{>0}$ such that
    \begin{align*}
        u_k(F_i^{-1}(x_{i,j,k}+r_{i,j,k}\cdot))\rightarrow\omega_{i,j}(\pi^{-1}(\cdot))&\text{  in }C^l_{\text{loc}}(B_R(0)\setminus S_{i,j})\text{ for all }l\geq0,R\geq0,\\
        x_{i,j,k}\rightarrow 0,&\\
        r_{i,j,k}\rightarrow0,&\\
        \textnormal{max}\{\frac{r_{i,j_1,k}}{r_{i,j_2,k}},\frac{r_{i,j_2,k}}{r_{i,j_1,k}},\frac{|x_{i,j_1,k}-x_{i,j_2,k}|}{r_{i,j_1,k}+r_{i,j_2,k}}\}\rightarrow\infty&\text{ whenever }j_1\neq j_2,
    \end{align*}
    where $\pi:\mathbb{S}^2\rightarrow \mathbb{R}^2$ is the stereographic projection and $S_{i,j}$ is the finite, possibly empty, set of points where further bubbles form on $\omega_{i,j}$.
    We then have the energy identity
    \begin{equation*}       \lim_{k\rightarrow\infty}E_{\varepsilon_k}[u_k]=E[u_\infty]+\sum_{i}^I\sum_j^{J(I)}E[\omega_{i,j}]+2\sum_{i}^I\sum_j^{J(I)}e^{-2\rho_i(0)}\mu_{i,j}\int_{\mathbb{R}^2}|\Delta \omega_{i,j}|^2dx,
    \end{equation*}
    where
    \begin{equation*}
        \mu_{i,j}=\lim_{k\rightarrow\infty}\frac{\varepsilon_k}{r_{i,j,k}^2}\log(1/r_{i,j,k}).
    \end{equation*}
\end{theorem}
We note here again that $r_{i,k}$ and $\omega_i$ are not uniquely defined however in this case our choice of $r_{i,k}$ affects $\mu_{i,k}$. We also have a term coming from the choice of isothermal coordinates. These occur, due to the fact that the parametrisation of the bubble affects the biharmonic energy. It is however easy to see that our choice does not change whether we have $\mu_i=0$ or not.

\vspace{5mm}

This also implies a stronger a priori bound for $\varepsilon$-harmonic maps. In Lamm's construction of the bubbling procedure in \cite{Lamm_epsilon} he shows that $\frac{\varepsilon_k}{r_k^2}\rightarrow0$. As an immediate consequence of Theorem \ref{theorem: generalised EI epsilon} we can improve this to the following.
\begin{corollary}
    In the setting of Theorem \ref{theorem: generalised EI epsilon} there exists $C\in(0,\infty)$ such that
    \begin{equation*}
        \frac{\varepsilon_k}{r_{i,k}^2}\log(1/r_{i,k})\leq C
    \end{equation*}
    for all $i$ and for all $k$ large enough.
\end{corollary}
In fact our proof of Theorem \ref{theorem: generalised EI epsilon} will be written such that it does not rely on the fact that $\frac{\varepsilon_k}{r_k^2}\rightarrow0$ so it can serve as an alternative proof that the bubble must be harmonic. This is a more direct way of seeing this fact.

This allows us to write down an intrinsic criterion for the energy identity which does not require a choice of bubbling radius.
\begin{corollary}
    Under the setup of Theorem \ref{theorem: generalised EI epsilon} we have the fully energy identity if and only if
    \begin{equation*}
        \varepsilon_k\int_M|\Delta u_k|^2\log(\int_M|\Delta u_k|^2)\rightarrow0
    \end{equation*}
    as $k\rightarrow\infty$.
\end{corollary}
\vspace{5mm}

This $\varepsilon$-energy identity is rather different from the $\alpha$ case.
One important difference to note is the location of the energy defect. In the $\varepsilon$ case we can only possibly lose energy on the neck region, whereas in the $\alpha$ case we can potentially lose energy on both the bubble and the neck.
Indeed Li and Wang show that in the limit the $\alpha$-energy is ${\mu}E[\omega]$ on the bubble and $({\mu}^2-\mu)E[\omega]$ on the neck and so energy is lost on the neck if and only if it is also lost on the bubble.

One question we can now ask is what happens to Dirichlet energy in the limit. In the $\varepsilon$-harmonic case this is the same thing as looking at the limit of $\varepsilon$-energy. In the $\alpha$ case we obtain a different identity. One can immediately notice this as we cannot have a Dirichlet energy defect on the bubble.
In fact by following the same steps as the proof of Theorem \ref{theorem: generalised EI epsilon} we obtain the following.
\begin{theorem}\label{theorem: generalised EI dirichlet alpha}
   Under the set-up of Theorem \ref{theorem: generalised EI alpha alpha} we have the following energy identity
    \begin{equation*}       \lim_{k\rightarrow\infty}E[u_k]=E[u_\infty]+\sum_{i}^I\sum_j^{J(I)}(1+2\log\mu_{i,j})E[\omega_{i,j}].
    \end{equation*}
\end{theorem}
This means that Dirichlet energy is conserved if and only if $\alpha$-energy is conserved. Indeed showing that the Dirichlet energy on the neck converges to 0 is sufficient for proving the energy identity. In particular this would allow us to automatically improved the result in \cite{li_zhu} and \cite{bayer_roberts} from a Dirichlet energy identity to an $\alpha$-energy identity. Though in these cases the no-neck property holds which also immediately implies the $\alpha$-energy identity.

\vspace{5mm}
This now leads us to consider the lengths of necks. Li and Wang show that there exists a special ratio that determines the length of the neck and that when it exists it contains a geodesic. We find the same result for $\varepsilon$-harmonic maps.
\begin{theorem}\label{theorem: neck result}
    Under the set-up of Theorem \ref{theorem: generalised EI epsilon} further assume that there is exactly one point of energy concentration with exactly one bubble forming $\omega\in C^\infty(\mathbb{S}^2,N)$ at rate $r_k$.
    By setting $\nu=\lim_k\sqrt{\frac{\varepsilon_k}{r_k^2}}\log(1/r_k)$ we find that there are three possible scenarios for the length of the neck.
    \begin{enumerate}
        \item If $\nu=0$ there is no neck. Explicitly
        \begin{equation*}
            \lim_{R\rightarrow\infty}\lim_{R_0\rightarrow0}\lim_{k\rightarrow\infty}\textnormal{osc}_{B_{R_0}(0)\setminus B_{Rr_k}(0) }(u_k\circ F^{-1})=0.
        \end{equation*}
        \item If $\nu\in(1,\infty)$ then the neck region consists of exactly a geodesic of length
        \begin{equation*}
            \nu\cdot\sqrt{\frac{e^{-2\rho(0)}}{\pi}\int_{\mathbb{R}^2}|\Delta \omega|^2dx}.
        \end{equation*}
        Explicitly we define the family of curves $\gamma_k:(0,1)\rightarrow \mathbb{R}^l$ by
        \begin{equation*}
            \gamma_k(t)=\frac{1}{2\pi r_k^t}\int_{\partial B_{r_k^t}(0)}u_k\circ F^{-1}.
        \end{equation*}
        Then $\gamma_k\rightarrow\Gamma$ in $C^0_{\text{loc}}((0,1))$ where $\Gamma\in C^\infty((0,1), N)$ is a geodesic of arc length, $\Gamma'(t)=\nu\cdot\sqrt{\frac{e^{-2\rho(0)}}{\pi}\int_{\mathbb{R}^2}|\Delta \omega|^2dx}$. Further, to see that $\Gamma$ is the whole neck,
        \begin{align*}
            \lim_{\alpha\rightarrow0}\lim_{R_0\rightarrow0}\lim_{k\rightarrow\infty}\textnormal{osc}_{B_{R_0}(0)\setminus B_{r_k^\alpha}(0) }(u_k\circ F^{-1})=&0,\\
            \lim_{\beta\rightarrow1}\lim_{R\rightarrow\infty}\lim_{k\rightarrow\infty}\textnormal{osc}_{B_{r_k^\beta}(0)\setminus B_{Rr_k}(0) }(u_k\circ F^{-1})=&0.
        \end{align*}
        \item If $\nu=\infty$ then the neck contains a geodesic of infinite length. $\gamma_k$, when reparametrised by arc length, converges locally uniformly to a geodesic of infinite length.
    \end{enumerate}
\end{theorem}
Note that the $\nu<\infty$ implies that $\mu=0$, so in the both the first two cases we have an exact $L^\infty$ and $W^{1,2}$ picture of our singularity model.
Following the work of Li and Wang \cite{Li_Wang_energy_counterexample}, Bayer and the author showed \cite{bayer_roberts} that there does indeed exist an example of a target manifold where the full energy identity can fail and there may exist a neck. 

\begin{center}
    Acknowledgements
\end{center}
The author would like to thank his supervisor Ben Sharp for many helpful discussions. The author is funded by Engineering and Physical Sciences Research Council (EPSRC) - EP/W524372/1, Studentship 2927009.

\section{Proof of the energy identity}\label{section: Proof of the energy identity}
We first recall the low energy regularity theorem proven by Lamm.
\begin{lemma}\label{lemma: regularity theorem}[\cite{Lamm_epsilon} Corollary 2.10]
    There exists $\delta_0>0$ and $C>0$ such that if $u_\varepsilon\in C^\infty(M,N)$ is $\varepsilon$-harmonic with $E_\varepsilon(u_\varepsilon,B_{32R}(x_0))<\delta_0$ for some $x\in\Sigma$ and $R>0$. Then for all $\varepsilon>0$ sufficiently small and any $k\in \mathbb{N}$
    \begin{equation*}
        \sum_{i=1}^kR^i\|\nabla^iu_\varepsilon\|_{L^\infty(B_R(x_0))}\leq C\sqrt{E_\varepsilon(u_\varepsilon,B_{32R}(x_0))}.
    \end{equation*}
\end{lemma}
We will now prove the generalised energy identity for $\varepsilon_k$-harmonic maps. We will prove this without assuming that $\frac{\varepsilon_k}{r_k^2}\rightarrow0$ and that $\omega$ is harmonic. Meaning that this proof can serve as an alternative proof of these facts. Indeed we will only assume that $\omega$ is a smooth, non trivial map $\mathbb{R}^2\rightarrow N$ as in Lamm's construction of the bubbling procedure in \cite{Lamm_epsilon}.
\begin{proof}[Proof of Theorem \ref{theorem: generalised EI epsilon}]
    We shall start by restricting to the case of only bubble. By using a standard covering argument it is easy to show this extends to the general case.
    Using local smooth convergence we immediately have
    \begin{equation*}
        \lim_{R\rightarrow0}\lim_{R_0\rightarrow\infty}\lim_{k\rightarrow\infty}E_{\varepsilon_k}[u_k;\{M\setminus{B_R(x)}\}\cup B_{R_0r_k}(x)]=E[u_\infty]+E[\omega]+\lim_{k\rightarrow\infty}(\frac{\varepsilon_k}{r_k^2})e^{2\rho(0)}\int_{\mathbb{R}^2}|\Delta \omega|^2dx
    \end{equation*}
    It remains to find the energy on the annulus.
    We first find the Hopf-differential type estimate akin to Lamm (Lemma 3.2 \cite{Lamm_minmax}).
    We first note the stress-energy tensor for the harmonic map and the stress-energy tensor biharmonic map constructed by Jiang \cite{Jiang}.
    \begin{align*}
        S^1_{\alpha\beta}(u):=&\frac{1}{2}|\nabla^gu|^2\delta_{\alpha\beta}-\langle\nabla^g_\alpha u,\nabla^g_\beta u\rangle,\\
        S^2_{\alpha\beta}(u):=&\frac{1}{2}|\Delta_gu|^2\delta_{\alpha\beta}+\langle\nabla^g_\gamma u,\nabla^g_\gamma\Delta_g u\rangle\delta_{\alpha\beta}-\langle\nabla^g_\alpha u,\nabla^g_\beta\Delta_g u\rangle-\langle\nabla^g_\beta u,\nabla^g_\alpha\Delta_g u\rangle.
    \end{align*}
    By taking the divergence we see that if $u$ is a $\varepsilon$-harmonic map then
    \begin{align*}
       \text{div}^g_\alpha( S^1_{\alpha\beta}(u)-\varepsilon S^2_{\alpha\beta}(u))=-\langle\nabla^g_\beta u,\Delta_gu-\varepsilon\Delta_g^2u\rangle=0.
    \end{align*}
    We can now differentiate by parts in a ball around 0 to get
    \begin{align*}
        \int_{B_r(0)}S^1_{\alpha\beta}(u)-\varepsilon S^2_{\alpha\beta}(u)\delta_{\alpha\beta}dx_g=\int_{\partial B_r(0)}(S^1_{\alpha\beta}(u)-\varepsilon S^2_{\alpha\beta}(u))x^\beta\nu_g^\alpha dS_g
    \end{align*}
    where $\nu_g$ is the outward unit normal with respect to the metric. Expanding out the terms this becomes
    \begin{equation}      
    \begin{split}\label{eq: E-L 1}
        \int_{\partial B_t}|\frac{\partial}{\partial r}u|^2-\frac{1}{t^2}|\frac{\partial}{\partial\theta}u|^2dS=& \frac{2\varepsilon}{t}\int_{B_t}e^{-2\rho}|\Delta u|^2dx 
       -2\varepsilon \int_{\partial B_t}\frac{e^{-2\rho}}{2}|\Delta u|^2\\&+\langle \nabla_i u,\nabla_i (e^{-2\rho}\Delta u)\rangle
       -2\langle \frac{\partial }{\partial r}u,\frac{\partial}{\partial r}(e^{-2\rho}\Delta u)\rangle dS.
    \end{split}
    \end{equation}
    We now may use the standard trick that given any $\delta>0$ we can find $R,K$ large enough and $R_0$ small enough such that $E_{\varepsilon_k}[u_k;B_{2a}\setminus B_a]<\delta$ for any $Rr_k<a<R_0/2$ and all $k\geq K$. This follows from the assumption that there is only one bubble.
    Now using Lemma \ref{lemma: regularity theorem} we have
    \begin{equation}\label{eq: pointwise small on annulus}
        |\nabla^t u_k|\leq C\sqrt{\delta}\frac{1}{|x|^t}
    \end{equation}
    for $t\in\mathbb{N}$ and $x\in A(2Rr_k,R_0/4):=B_{R_0/4}\setminus B_{2Rr_k}$.
    Then by integrating \eqref{eq: E-L 1} and using \eqref{eq: pointwise small on annulus} we obtain an equation for energy on the annulus.
    \begin{align}\nonumber
        |E_{\varepsilon_k}[u_{k};A(2Rr_k,R_0/4)]-2\varepsilon_k\int_{Rr_k}^{R_0/4}(\frac{1}{r}\int_{B_r}e^{-2\rho}|\Delta u_k|^2dx)dr|
        \\\leq C\int_{A(2Rr_k,R_0/4)}\frac{1}{r^2}|\frac{\partial}{\partial \theta}u_k|^2dx+C\varepsilon_k\int_{A(2Rr_k,R_0/4)}|\Delta u_k|^2+|\nabla u_k|(|\nabla^3u_k|+|\nabla^2u_k|)dx\label{eq: sharp inital bound on whole annulus} \\\nonumber \leq C(\frac{\varepsilon_k}{r_k^2}+\sqrt{\delta}).
    \end{align}
    Where we have used Lemma 3.3 in \cite{Lamm_minmax} to get the following bound on the tangential derivative
    \begin{equation*}
        \int_{A(2Rr_k,R_0/4)}\frac{1}{r^2}|\frac{\partial}{\partial \theta}u_k|^2dx\leq C(\frac{\varepsilon_k}{r_k^2}+\sqrt{\delta}).
    \end{equation*}
    Note that this constant will be independent of $R$ and $R_0$ as increasing $R$ and decreasing $R_0$ decreases the size of the annulus.
    We then note that
    \begin{equation}\label{eq: biharmonic on ball}
        r_k^2\int_{B_{Rr_k}}e^{-2\rho}|\Delta u_k|^2dx\rightarrow e^{-2\rho(0)}\int_{B_R}|\Delta \omega|^2dx
    \end{equation}
    and for $Rr_k<r<R_0/4$ we have, using \eqref{eq: pointwise small on annulus},
    \begin{equation}\label{eq: biharmonic on annulus}
        \int_{B_r\setminus B_{Rr_k}}e^{-2\rho}|\Delta u_k|^2dx\leq C\sqrt{\delta}\int_{\mathbb{R}^2\setminus B_{Rr_k}}\frac{1}{|x|^4}dx\leq C\sqrt{\delta}\frac{1}{r_k^2}.
    \end{equation}
    So we have 
    \begin{equation}\label{eq: energy full annulus, nearly done}
    \begin{split}
        E_{\varepsilon_k}[u_{k};A(2Rr_k,R_0/4)]-{\frac{\varepsilon_k}{r_k^{2}}}\log (1/r_k)(2e^{-2\rho(0)}\int_{B_R}|\Delta \omega|^2dx+o_k(1)+\mathcal{O}_k(\sqrt{\delta}))\\=
        \mathcal{O}_k(\frac{\varepsilon_k}{r_k^2})(1+\int_{B_R}|\Delta \omega|^2dx)
        +\mathcal{O}_k(\sqrt{\delta}).
    \end{split}
    \end{equation}
    We know that $\int_{\mathbb{R}^2}|\Delta \omega|^2>0$ as else $\omega$ would be a harmonic function from $\mathbb{R}^2$ into a compact subset of $\mathbb{R}^l$ so constant, but $\omega$ is non-constant by definition. Now by taking $k$ large enough, $R$ large enough and $R_0$ small enough we may ensure that
    \begin{equation*}
        2e^{-2\rho(0)}\int_{B_R}|\Delta \omega|^2dx+o_k(1)+\mathcal{O}_k(\sqrt{\delta})\geq C>0.
    \end{equation*}
    Our uniform bound on energy now entails that ${\frac{\varepsilon_k}{r_k^{2}}}\log (1/r_k)$ must be uniformly bounded and so ${\frac{\varepsilon_k}{r_k^{2}}}\rightarrow0$.
    By examining the Euler-Lagrange equations we now see that $\omega$ must be harmonic. $\omega$ has finite energy so we may extend over the singularity at $\infty$. This means $\omega$ is a map from the sphere so $\int_{\mathbb{R}^2}|\Delta \omega|^2dx$ must be finite.
    Now taking $k$ to infinity gives
    \begin{equation}\label{eq: energy on full annulus final bound}
        E_{\varepsilon_k}[u_{k};A(2Rr_k,R_0/4)]-\lim_{k\rightarrow\infty}(\log (1/r_k){\frac{\varepsilon_k}{r_k^{2}}})(2e^{-2\rho(0)}\int_{B_R}|\Delta \omega|^2dx+\mathcal{O}(\sqrt{\delta}))=\mathcal{O}(\sqrt{\delta}).
    \end{equation}
    By taking $R\rightarrow\infty$, $R_0\rightarrow0$ we may send $\delta$ to 0 completing the proof.
\end{proof}

The proof of the Dirichlet energy identity for $\alpha$-harmonic maps is essentially identical.
\begin{proof}[Proof of Theorem \ref{theorem: generalised EI dirichlet alpha}]    
The Pohozaev identity for $\alpha$-harmonic maps (Lemma 2.4 \cite{li_wang_general_EI}) tells us that
\begin{align*}
    \int_{A(2Rr_k,R_0/4)}|\nabla u|^2dx=&-2(\alpha -1)\int_{2Rr_k}^{R_0/4}\frac{1}{t}\int_{B_t}\frac{\nabla_i|\nabla u_k|^2\nabla_i u_k}{1+|\nabla u_k|^2}r\frac{\partial u_k}{\partial r}dxdt\\&+2\int_{A(2Rr_k,R_0/4)}\frac{1}{r^2}|\frac{\partial u_k}{\partial \theta}|^2dx.
\end{align*}
One can again show that the tangential derivative is small for $\alpha$-harmonic maps.
In particular we have (Lemma 2.5 \cite{Lamm_minmax})
\begin{equation*}
    \int_{A(2Rr_k,R_0/4)}\frac{1}{r^2}|\frac{\partial u_k}{\partial \theta}|^2dx\leq C\sqrt{\delta}(1+(\alpha_k-1)\log (1/r_k))\leq C\sqrt{\delta},
\end{equation*} where $\delta$ is defined as in the $\varepsilon$ case. Using the same approach as the $\varepsilon$ case it is easy to see that
\begin{equation}
    \lim_kE[u_{k}]=E[u_\infty]+E[\omega]-2\lim_k((\alpha_k-1)\log (1/r_k))\int_{\mathbb{R}^2}\frac{\nabla_i|\nabla \omega|^2\nabla_i \omega}{|\nabla \omega|^2}r\frac{\partial \omega}{\partial r}dx.
\end{equation}
We note that $\omega$ is harmonic and in particular conformal, so we find
\begin{align*}
    \int_{\mathbb{R}^2}\frac{\nabla_i|\nabla \omega|^2\nabla_i \omega}{|\nabla \omega|^2}r\frac{\partial \omega}{\partial r}dx
    =-E[\omega].
\end{align*}
\end{proof}

\section{Proof of the neck result}
We will follow the approach of Li and Wang to show that $\nu=0$ implies the no neck property.
The proof for the $\nu>0$ case similarly follows from Li and Wang's proof by making the same changes as below.
\begin{proof}[Proof Theorem \ref{theorem: neck result}]

Note here the Euler-Lagrange equation for $\varepsilon$-harmonic maps
\begin{equation}\label{eq: full EL}
    0=(\Delta_g u-\varepsilon\Delta_g^2u)^\top=\Delta_g u-A(u)(\nabla^g u,\nabla^g u)+\varepsilon (\Delta_g^2u)^\top.
\end{equation}
We will now define $u^*_k(r)=u^*_k(|x|)$ by
\begin{equation*}
    u^*_k(r)=\frac{1}{2\pi r}\int_{\partial B_r}u_kds=\frac{1}{2\pi}\int_0^{2\pi}u_k(r,\theta)d\theta.
\end{equation*}
Now testing the Euler-Lagrange equation with $u_k-u_k^*$ we get, for any annulus $A(p,q)$,
\begin{align}\label{eq: E-L 2}
    \int_{A(p,q)}|\nabla u_k|^2dx\leq&\int_{A(p,q)}A(u_k)(\nabla u_k,\nabla u_k)\cdot(u_k-u^*_k)+\nabla u_k\cdot\nabla u^*_kdx\nonumber\\
    &+\int_{\partial A(p,q)}\frac{\partial u_k}{\partial r}(u_k-u_k^*)ds+\varepsilon_k\int_{A(p,q)}(\Delta^2_gu)^\top\cdot(u_k-u^*_k)dx_g.
\end{align}
As in \cite{li_wang_general_EI} we can bound,
\begin{equation*}
    \begin{split}
    \int_{A(p,q)}\nabla u_k\cdot\nabla u^*_kdx&=
    \int_{A(p,q)}\frac{\partial u_k}{\partial r}\cdot\frac{\partial u_k^*}{\partial r}\\
    &\leq(\int_{A(p,q)}\vert\frac{\partial u_k}{\partial r}\vert^2dx\cdot\int_{A(p,q)}\vert\frac{\partial u_k^*}{\partial r}\vert^2dx)^{1/2}\\
    &\leq\int_{A(p,q)}\vert\frac{\partial u_k}{\partial r}\vert^2dx.
    \end{split}
\end{equation*}
Then by a Poincar\'e inequality we have
\begin{equation*}
    \int_{\partial B_t}|u_k-u_k^*|^2\leq \int_{\partial B_t}|\frac{\partial u_k}{\partial\theta}|^2,
\end{equation*}
giving
\begin{equation*}
\begin{split}
    |\int_{\partial B_t}\frac{\partial u_k}{\partial r}(u_k-u^*_k)ds|\leq&(\int_{\partial B_t}|\frac{\partial u_k}{\partial r}|^2\int_{\partial B_t}|\frac{\partial u_k}{\partial\theta}|^2)^{1/2}\\
    &\leq\frac{1}{2}\int_{\partial B_t}t|\frac{\partial u_k}{\partial r}|^2+\frac{1}{t}|\frac{\partial u_k}{\partial \theta}|^2\\
    &=\frac{1}{2}\int_{\partial B_t}t|\nabla u_k|^2ds .
\end{split}
\end{equation*}
We can use \eqref{eq: pointwise small on annulus} to ensure that $u_k-u^*_k$ is pointwise small on our annulus so given $\gamma>0$ we have
\begin{equation}\label{eq: 2FF bounded by gamma}
    \int_{A(p,q)}A(u_k)(\nabla u_k,\nabla u_k)(u_k-u^*_k)dx\leq \gamma\int_{A(p,q)}|\nabla u_k|^2dx
\end{equation}
for $R_0$ large enough, $R$ small enough, $k$ large enough and any $R_0r_k<p<q<R$.
Now by integrating $\eqref{eq: E-L 1}$ we get
\begin{equation}\label{eq: better r bound}
    \int_{A(p,q)}|\frac{\partial}{\partial r}u|^2\leq\frac{1}{2}\int_{A(p,q)}|\nabla u_k|^2+C\varepsilon_k\Big(\int_{p}^q\frac{1}{r}\int_{B_r}|\Delta_g u|^2dx_gdr+\int_{A(p,q)}|\nabla u_k|(|\nabla^3u_k|+|\nabla^2u_k|)\Big).
\end{equation}
Now combining \eqref{eq: E-L 2}, \eqref{eq: 2FF bounded by gamma} and \eqref{eq: better r bound} we get
\begin{equation}\label{eq: gamma energy on annulus}
\begin{split}
    (\frac{1}{2}-\gamma)\int_{A(p,q)}|\nabla u_k|^2dx\leq& C\varepsilon_k\Big(\int_{A(p,q)}|(\Delta^2u_k)^\top|+|\nabla u_k|(|\nabla^3u_k|+|\nabla^2u_k|)dx_g\\
    &+\int_{p}^q\frac{1}{r}\int_{B_r}|\Delta_g u|^2dx_gdr\Big)+\frac{1}{2}\int_{\partial A(p,q)}r|\nabla u_k|^2ds
\end{split}
\end{equation}
again for any $\gamma>0$, $R_0$ large enough, $R$ small enough, $k$ large enough and any $R_0r_k<p<q<R$.
We then have, using \eqref{eq: pointwise small on annulus}, \eqref{eq: biharmonic on ball} and \eqref{eq: biharmonic on annulus}
\begin{equation*}
    \int_{A(p,q)}|(\Delta_g^2u_k)^\top|+|\nabla u_k|(|\nabla^3u_k|+|\nabla^2u_k|)dx_g+\int_{p}^q\frac{1}{r}\int_{B_r}|\Delta_g u|^2dx_gdr\leq C\frac{1}{p^2}(1+\log q/p).
\end{equation*}
Now setting $p(t)=2^{l-t}Rr_k$ and $q(t)=2^{l+t}Rr_k$ and $f_k(t)=\int_{A(p(t),q(t))}|\nabla u_k|^2$ for any $l,t$ s.t. $Rr_k<p(t)<q(t)<R_0$ we get
\begin{equation*}
    (\frac{1}{2}-\gamma)f_k(t)\leq \frac{1}{2\log2}f_k'(t)+C\frac{\varepsilon_k}{r_k^2}(1+t).
\end{equation*}
Multiplying this differential inequality by $2^{1-(1-2\gamma)t}$ and integrating gives
\begin{equation*}
    [-\frac{1}{\log2}2^{-(1-2\gamma)t}f_k(t)]_{t_0}^{t_1}\leq [-C\frac{\varepsilon_k}{r_k^2}2^{-(1-2\gamma)t}\frac{1+(1-2\gamma)(\log2)(1+t)}{(1-2\gamma)^2(\log2)^2}]_{t_0}^{t_1}.
\end{equation*}
Setting $t_0=1$ and $\gamma=\frac{1}{4}$ gives, for $t>1$
\begin{equation*}
    f_k(1)\leq C(2^{-t/2}f_k(t)+\frac{\varepsilon_k}{r_k^2}).
\end{equation*}
This tells us that we have some sort of decay of the energy.
Now set $P=(\log R_0-\log Rr_k)/\log2$, $1\leq l\leq \lfloor{P}\rfloor-1$ then $t=L_l:=\min\{l,P-l\}$. Now
\begin{equation*}
    E_{\varepsilon_k}[u_k;A(2^{l-1}Rr_k,2^{l+1}Rr_k)]\leq f_k(1)+C\frac{\varepsilon_k}{r_k^2}\leq C2^{-L_l/2}E[u_k,A(Rr_k,R_0)]+C\frac{\varepsilon_k}{r_k^2}.
\end{equation*}
Then using the low energy regularity lemma, Lemma \ref{lemma: regularity theorem} we get
\begin{equation*}
\begin{split}
    \text{osc}_{A(2^{l-1/2}Rr_k,2^{l+1/2}Rr_k)}(u_k)\leq C2^{-L_l/4}\sqrt{E[u_k,A(Rr_k,R_0)]}+C\sqrt{\frac{\varepsilon_k}{r_k^2}}.
\end{split} 
\end{equation*}
By summing over these annuli, noting that $P\leq C(\log(1/r_k)+1)$, we get
\begin{equation*}
\begin{split}
    \text{osc}_{A(2Rr_k,R_0/4)}(u_k)\leq&
    C\sqrt{E[u_k,A(Rr_k,R_0)]}\sum_{l=1}^{\lfloor{P}\rfloor-1}2^{-L_l/4}+CP\sqrt{\frac{\varepsilon_k}{r_k^2}}\\
    \leq&C\sqrt{E[u_k,A(Rr_k,R_0)]}+C(\log(1/r_k)+1)\sqrt{\frac{\varepsilon_k}{r_k^2}}
    \\ &\rightarrow0  \text{   as }k\rightarrow\infty.
\end{split}
\end{equation*}
Where we have used the fact that $\nu=0$ implies $\mu=0$.
This completes the $\nu=0$ case.
\end{proof}



\bibliographystyle{abbrv}
\bibliography{Bibliography}

\end{document}